\documentclass{article}
\usepackage{amssymb, amsmath, amsthm}
\usepackage{graphicx,verbatim}
\usepackage{rotating}
\usepackage[all,cmtip]{xy}
\usepackage{url}
\usepackage{mathrsfs}
\usepackage{bbold} 
\usepackage[colorlinks = true]{hyperref}
\usepackage[normalem]{ulem}
\usepackage{enumitem}
\usepackage[colorinlistoftodos]{todonotes}
\usepackage{appendix}

\usepackage[a4paper,left=1in,width=458pt,textheight=650pt,top=1.5in]{geometry}

\usepackage{tikz}
\usetikzlibrary{calc}
\usepgflibrary{shapes.geometric}
\usepgflibrary{shapes.misc}
\usetikzlibrary{positioning}
\usetikzlibrary{decorations}
\usetikzlibrary{arrows}

    \theoremstyle{plain}
    \newtheorem{theorem}{Theorem}
    \newtheorem{proposition}[theorem]{Proposition}
    \newtheorem{corollary}[theorem]{Corollary}
    \newtheorem{lemma}[theorem]{Lemma}
    \newtheorem{conjecture}{Conjecture}
    \newtheorem*{theorem*}{Theorem}
    \newtheorem*{proposition*}{Proposition}
    \newtheorem*{corollary*}{Corollary}
    \newtheorem*{lemma*}{Lemma}
    \newtheorem*{conjecture*}{Conjecture}

    \theoremstyle{definition}
    \newtheorem{definition}[theorem]{Definition}
    \newtheorem{example}[theorem]{Example}
    \newtheorem*{definition*}{Definition}
    \newtheorem*{example*}{Example}

    \theoremstyle{remark}
    \newtheorem{remark}[theorem]{Remark}
    \newtheorem*{remark*}{Remark}
    \newtheorem*{question*}{Question}

    \newcommand{\cF}{\mathcal{F}}
    \newcommand{\cG}{\mathcal{G}}

    \newcommand{\cK}{\mathcal{K}}
    
    \newcommand{\cH}{\mathcal{H}}

    \newcommand{\cP}{\mathcal{P}}

    \newcommand{\im}{\operatorname{im }}
    \renewcommand{\ker}{\operatorname{ker }}

    \newcommand{\Hom}{\operatorname{Hom}}
    \newcommand{\Mor}{\operatorname{Mor}}
    \newcommand{\Obj}{\operatorname{Obj}}

    \newcommand{\op}[1]{{#1}^{\textup{op}}}

    \newcommand{\id}{\operatorname{id}}

     \newcommand{\defeq}{\stackrel{\scriptscriptstyle{\mathrm{def}}}{=}}

    \renewcommand{\iff}{\Leftrightarrow}

    \newcommand{\git}{\mathbin{
  \mathchoice{/\mkern-6mu/}
    {/\mkern-6mu/}
    {/\mkern-5mu/}
    {/\mkern-5mu/}}}
    \renewcommand\bar\overline
    \renewcommand\tilde\widetilde
    \newcommand{\smor}[1]{\stackrel{#1}{\nabla}}

\title{Slices of groupoids are group-like}
\author{Nicholas Cooney
\and  Jan E. Grabowski\footnotemark[3] 
}
\date{\today}

\begin{document}
\maketitle

\renewcommand{\thefootnote}{\fnsymbol{footnote}}
\footnotetext[3]{Email: \url{j.grabowski@lancaster.ac.uk}.  Website: \url{http://www.maths.lancs.ac.uk/~grabowsj/}}
\renewcommand{\thefootnote}{\arabic{footnote}}
\setcounter{footnote}{0}

\begin{abstract}
  Given a category, one may construct slices of it.  That is, one builds a new category whose objects are the morphisms from the category with a fixed codomain and morphisms certain commutative triangles.  If the category is a groupoid, so that every morphism is invertible, then its slices are (connected) groupoids.
  
  We give a number of constructions that show how slices of groupoids have properties even closer to those of groups than the groupoids they come from.  These include natural notions of kernels and coset spaces.

  \vspace{1em}
\end{abstract}


\section{Introduction}

  A groupoid is a category in which every morphism is invertible.  Since a one-object groupoid is naturally identified with a group, groupoids are regarded as a many-object generalisation of a group.  However groupoids do not have a completely parallel theory to that of groups.  For example, a groupoid does not have an identity element.

  In this short note, we observe that by passing from a groupoid to one of its slices, one finds a situation closer to that of groups.  A slice of a category $\mathcal{C}$ at an object $X$ is a category whose objects are morphisms of $\mathcal{C}$ with codomain $X$ and morphisms certain commutative triangles.

  In particular, since any object has an identity map $\id_{X}\colon X\to X$, the slice of a groupoid has a distinguished object $\id_{X}$ which, of course, has the properties of an identity.

  Given a functor $\cF$ between groupoids $\cG$ and $\cH$, we have an induced functor between $\cG/X$ and $\cH/\cF X$.  With some relatively mild assumptions, this induced functor behaves like a group homomorphism: we may define its kernel, prove an analogue of the correspondence between cosets of a kernel and elements of the image and show that there is a natural bijection between pre-images.

  In a slightly different direction, we give a ``coset space'' construction associated to the action of a subgroupoid on either the whole groupoid or a slice of it.

  These constructions were found in the course of work on the paper \cite{QuantumAutomorphisms}. However they are much more general than the setting of that work and we hope they may be of independent interest.

\section{Identities and kernels}

Let $\cG$ be a groupoid.  We first recall the definition of the slice groupoid, $\cG \slash X$, for $X \in \cG$.  The objects of $\cG \slash X$ are all morphisms in $\cG$ with codomain $X$.  A morphism from $f\colon Y\to X$ to $f'\colon Z \to X$ is a commutative triangle
\begin{center}
\begin{tikzpicture}[node distance=1cm,on grid]
\node (B) at (0,0) {$Y$};
\node (C) [right=of B] {$Z$};
\node (A) [below=of B,xshift=0.5cm,yshift=0cm] {$X$};
\draw[->] (B) -- node [above] {$g$} (C);
\draw[->] (B) -- node [left] {$f$} (A);
\draw[->] (C) -- node [right] {$f'$} (A);
\end{tikzpicture}
\end{center}
when a suitable $g\in \Mor(\cG)$ exists.  Let us write $\smor{g}\colon f \to f'$ for such a triangle.

Composition of morphisms is given by horizontal concatenation of triangles:
\begin{center}
  \begin{tikzpicture}[node distance=1cm,on grid]
    \begin{scope}[xshift=-4cm]
      \node (B) at (0,0) {$Y$};
\node (C) [right=of B] {$Y'$};
\node (A) [below=of B,xshift=0.5cm,yshift=0cm] {$X$};
\draw[->] (B) -- node [above] {$g$} (C);
\draw[->] (B) -- node [left] {$f$} (A);
\draw[->] (C) -- node [right] {$f'$} (A);
    \end{scope}
    \node (circ) at (-2.5,-0.5) {$\circ$};
    \begin{scope}[xshift=-2cm]
      \node (B) at (0,0) {$Y'$};
\node (C) [right=of B] {$Y''$};
\node (A) [below=of B,xshift=0.5cm,yshift=0cm] {$X$};
\draw[->] (B) -- node [above] {$g'$} (C);
\draw[->] (B) -- node [left] {$f'$} (A);
\draw[->] (C) -- node [right] {$f''$} (A);
    \end{scope}
    \node (eq) at (0,-0.5) {$=$};
    \begin{scope}[xshift=1cm]
      \node (B) at (0,0) {$Y$};
\node (C) [right=of B] {$Y'$};
\node (D) [right=of C] {$Y''$};
\node (A) [below=of C] {$X$};
\draw[->] (B) -- node [left,xshift=-0.25em] {$f$} (A);
\draw[->] (C) -- node [right,xshift=-0.1em] {$f'$} (A);
\draw[->] (D) -- node [right] {$f''$} (A);
\draw[->] (B) -- node [above] {$g$} (C);
\draw[->] (C) -- node [above] {$g'$} (D);
    \end{scope}
        \node (eq2) at (4,-0.5) {$=$};
        \begin{scope}[xshift=5cm]
                \node (B) at (0,0) {$Y$};
\node (C) [right=of B] {$Y''$};
\node (A) [below=of B,xshift=0.5cm,yshift=0cm] {$X$};
\draw[->] (B) -- node [above,yshift=0.1em] {$g'\circ g$} (C);
\draw[->] (B) -- node [left] {$f$} (A);
\draw[->] (C) -- node [right] {$f''$} (A);
    \end{scope}
    
\end{tikzpicture}
\end{center}
which we will write in our more compact notation as $\smor{g} \circ \smor{g'}=\smor{g' \circ g}$.  This is the natural order to write composition in the slice category, although it conflicts with the ``right to left'' convention for composing morphisms.

Note that since $\cG$ is a groupoid, for any $f,f'$, there exists a morphism $f\to f'$ in $\cG \slash X$ induced by $g=(f')^{-1}\circ f$.  Moreover, since every morphism in $\cG$ is invertible, the every morphism in $\cG \slash X$ is too, so that $\cG \slash X$ is a connected groupoid.  In fact, more is true.

\begin{lemma} In the slice groupoid $\cG/X$, we have $|\Hom_{\cG/X}(f,f')|=1$ for all objects $f,f'$ of $\cG/X$.
\end{lemma}

\begin{proof} A morphism from $f$ to $f'$ in the slice category is a triangle $\smor{g}$ with $g$ such that $f=f'\circ g$.  If $\smor{g'}$ is another morphism, so $f=f'\circ g'$, then since we are in a groupoid, $f'\circ g=f=f'\circ g'$ implies $g=g'$.
  \end{proof}

We also see that $\id_{X}\colon X\to X$ is a zero object in $\cG \slash X$, that is, it is both initial and terminal.  For we have the morphisms $\smor{g^{-1}}\colon \id_{X} \to g$ and $\smor{g}\colon g\to \id_{X}$, which are unique by uniqueness of inverses.  We will shortly show that $\id_{X}$ is the natural analogue of the identity element of a group.  Indeed, if $\cG$ is a one-object groupoid with object $\ast$, then $\Mor_{\cG}(\ast,\ast)$ is a group with identity element $\id_{\ast}\colon \ast \to \ast$.  Considered in the slice $\cG/\ast$, this element is precisely the zero object just discussed.

Now let $\cH$ be another groupoid and consider $\cF\colon \cG \to \cH$ a functor.  For each $X$, there is a induced functor $\cF_{X}\colon \cG/X \to \cH/\cF X$ given on objects of $\cG/X$ by $\cF_{X}f \defeq\cF f$ and on morphisms by $\cF_{X}\smor{g}\ \defeq\ \smor{\cF g}$.  Note that the codomain of $\cF_{X}f=\cF f$ is indeed $\cF X$ and also that $\smor{\cF g}$ is precisely the commutative triangle in $\cH$ given by taking the image of $\smor{g}$ under $\cF$.

Note that $\cF_{X}$ is necessarily essentially surjective, since $\cH/\cF X$ is a connected groupoid; in the groupoid setting, ``up to isomorphism'' statements are extremely weak.

Let us denote by $\cG_{X}$ the connected component of $\cG$ containing $X$. The connected component $\cH_{\cF X}$ consists of all objects $Y$ of $\cH$ such that there exists a morphism $h\colon Y \to \cF X$.  (Since $\cH$ is a groupoid, it is not necessary to also consider the objects $Z$ for which there is a morphism $h'\colon \cF X \to Z$, as such a morphism exists if and only if $(h')^{-1}\colon Z \to \cF X$ exists.)

Note that connected components are, by definition, taken to be full subcategories and are therefore also groupoids.  Note too that the slices $\cG/X$ and $\cG_{X}/X$ are naturally identified: taking the slice at $X$ disregards any morphism whose domain is not in the connected component containing $X$.  This shows that if we are interested only in slices of groupoids, we may assume our groupoids are connected without loss of generality.

If $X'$ is an object in $\cG_{X}$, so that there exists a morphism $g\colon X' \to X$, then $\cF X'$ is an object of $\cH_{\cF X}$, as witnessed by $\cF g$.  Then if $g\colon X' \to X''$ is a morphism in $\cG_{X}$, $\cF g\colon \cF X' \to \cF X''$ is a morphism in $\cH_{\cF X}$, since it is a morphism between objects of $\cH_{\cF X}$ and the latter is a full subcategory of $\cH$.

  We may therefore consider the restriction of $\cF$ to $\cG_{X}$ as a functor $\cF|_{\cG_{X}}\colon \cG_{X} \to \cH_{\cF X}$. Since $\cH_{\cF X}$ is a connected groupoid, $\cF|_{\cG_{X}}$ is also essentially surjective.


The slice $\cH/\cF X$ has its own identity element, the zero object of $\cH/\cF X$, $\id_{\cF X}$.  It is then a natural question whether or not we have $\cF \id_{X}=\id_{\cF X}$.  From the lemma below, we see that this holds provided the functor $\cF|_{\cG_{X}}$ is ``full at $X$'' (it need not be full, although this is certainly sufficient).

\begin{lemma}\label{l:img-of-id-is-id} Assume that the restriction $\cF|_{\cG_{X}}$ of $\cF$ to the connected component $\cG_{X}$ has the property that the induced function $\Mor_{\cG_{X}}(X,X)\to \Mor_{\cH_{\cF X}}(\cF X,\cF X)$ is surjective. Then $\cF \id_{X}=\id_{\cF X}$.
\end{lemma}

\begin{proof} Let $g\in \Mor_{\cH}(\cF X,\cF X)$.  By the assumption, there exists $f\in \Mor_{\cG}(X,X)$ such that $g=\cF f$.  Since for all $f\in \Mor_{\cG}(X,X)$, $\id_{X}$ satisfies $\id_{X}\circ f=f=f\circ \id_{X}$, we have that $\cF \id_{X} \circ\, g=g=g\circ \cF \id_{X}$ for all $g\in \Mor_{\cH}(\cF X,\cF X)$.  Hence $\cF \id_{X}$ is an identity element in the group $\Mor_{\cH}(\cF X,\cF X)$ and so is equal to $\id_{\cF X}$.
\end{proof}

Now we may define a notion of kernel for functors between slices of groupoids.

\begin{definition} Let $\cG$ and $\cH$ be groupoids, $\cF\colon \cG \to \cH$ a functor between them and $X$ an object of $\cG$. Let $\cF_{X}\colon \cG/X \to \cH/\cF X$ be the induced functor.

  Define
  \[ \ker \cF_{X} = \{ f\colon Y\to X \mid \cF_{X}f=\id_{\cF X} \} \]
  to be the collection of objects of $\cG/X$ whose image under $\cF_{X}$ is the object $\id_{\cF X}$ of $\cH/\cF X$.
\end{definition}

The kernel of $\cF_{X}$ may in principle be an empty collection, but provided $\cF$ is full at $X$ (i.e. satisfies the assumption in Lemma~\ref{l:img-of-id-is-id}), we have $\id_{X}\in \ker \cF_{X}$ and a non-empty kernel.

One may easily verify the following properties of objects of the the kernel.
\begin{itemize}
\item For any $f\colon Y\to X$ belonging to $\ker \cF_{X}$, we have $\cF Y=\cF X$.
\item Given $f\in \ker \cF_{X}$, the morphism $\smor{f}\in \Mor_{\cG/X}(f,\id_{X})$ satisfies $\cF_{X} \smor{f}=\smor{\id_{\cF X}}$ (noting that $\smor{\id_{\cF X}}\colon \id_{\cF X}\to \id_{\cF X}$ is the identity morphism for the object $\id_{\cF X}$ in $\cH/\cF X$).
\item More generally, given $f,g\in \ker \cF_{X}$, we have a (unique) morphism in $\cG/X$ given by $\smor{g^{-1}\circ f}\colon f \to g$ and this morphism satisfies $\cF_{X} \smor{g^{-1}\circ f}=\smor{\id_{\cF X}}$.
\end{itemize}
Note that in the last of these properties, $g^{-1}\circ f$ is not necessarily itself in $\ker \cF_{X}$, since the codomain of $g^{-1}\circ f$ is not necessarily $X$.  

Let us denote by $\im \cF_{X}$ the image of $\cF_{X}$, as a collection of objects of $\cH/\cF X$.  (Note that at present, both $\ker \cF_{X}$ and $\im \cF_{X}$ are not being considered as categories, only as subcollections of the objects of the respective groupoids.)  The following is immediate.

\begin{lemma} Let $\cG$ and $\cH$ be groupoids, $\cF\colon \cG \to \cH$ a functor between them and $X$ an object of $\cG$. Let $\cF_{X}\colon \cG/X \to \cH/\cF X$ be the induced functor.

  For any $g\in \im \cF_{X}$, define
  \[ \cF_{X}^{-1}(g) = \{ f\colon Y\to X \mid \cF f=g \}. \]
  
  Then
  \begin{enumerate}[label=\textup{(\roman*)}]
  \item the collection $\cP(\cF_{X})\defeq \{ \cF_{X}^{-1}(g) \mid g\in \im \cF_{X} \}$ is a partition of the collection of objects of $\cG/X$; and
    \item there is a bijection between $\cP(\cF_{X})$ and $\im \cF_{X}$.
  \end{enumerate} \qed
\end{lemma}



Observe that if $\cG$ and $\cH$ are one-object groupoids, the above definitions and results reduce to exactly the classical constructions and statements for groups.

\section{Coset spaces for groupoids}

We will now explain how to construct the groupoid version of a coset space of a group with respect to a subgroup, and the natural group action on this.  

\subsubsection*{The action groupoid}

Let $\cG$ be a groupoid. There is a groupoid $\cG \git \cG$, the action groupoid with respect to the (right) action of $\cG$ on itself, via the functor $\Hom_{\cG}(?,-)\colon \cG \to \mathrm{Set}$ given by
\begin{align*} & \Hom_{\cG}(?,-)(X) = \Hom_{\cG}(X,-)\defeq \bigsqcup_{Y\in \cG} \Hom_{\cG}(X,Y) \\
  & \Hom_{\cG}(?,-)(f\colon X\to Y)\colon \Hom_{\cG}(Y,-)\to \Hom_{\cG}(X,-), (g\colon Y\to Z) \mapsto (g\circ f \colon X \to Z)
\end{align*}
Let us write $\rho_{\cG}$ for $\Hom_{\cG}(?,-)$.

The groupoid $\cG \git \cG$ has objects
\[ f\colon X\to Y \in \bigsqcup_{X} \bigsqcup_{Y} \Hom_{\cG}(X,Y) = \Mor(\cG) \]
and morphisms
\[ (g\colon Y \to Z, f\colon X \to Y)\colon g \to g\circ f. \]

Composition in $\cG \git \cG$ is given by
\[ \left((g',f') \circ (g,f) \colon g \to g\circ f \circ f'\right) = (g,f\circ f') \]
where this is defined, i.e.\ when $g'=g\circ f$. We have $\id_{g}=(g,\id_{s(g)})$.

We have that slice groupoids embed naturally in the action groupoid, as follows.

\begin{lemma}
There is a fully faithful contravariant functor that is injective on objects, $\iota_{X} \colon \cG \slash X \to \cG \git \cG$ given by
\begin{align*} & \iota_{X}(f\colon Y \to X) = f \\
  & \iota_{X}(\smor{g}\colon f \to f') = (f',g)\colon f' \to f
\end{align*}
\end{lemma}

\begin{proof} For functoriality, let $\smor{g}\colon f\to f'$, $\smor{g'}\colon f' \to f''$.  Then we have
  \begin{align*} \iota_{X}(\smor{g} \circ \smor{g'}) & =\iota_{X}(\smor{g' \circ g}) \\
      & = (f'',g'\circ g)\colon f'' \to (f''\circ g' \circ g =f) \\
      & = (f',g)\circ (f'',g') \\
      & = \iota_{X}(\smor{g})\circ \iota_{X}(\smor{g'}).
  \end{align*}
It is straightforward to see that $\iota_{X}$ respects identity morphisms.  Despite appearances, this shows that $\iota_{X}$ is a \emph{contravariant} functor: our notational convention for composition in the slice category is unfortunately misleading at this point.
  
The functor $\iota_{X}$ is faithful as if $\iota_{X}(\smor{g}\colon f \to f')=\iota_{X}(\smor{h}\colon f \to f')$, then $(f',g)=(f',h)$.  Then necessarily $g=h$.

For fullness, take $(f',g)\colon \iota_{X}(f')\to \iota_{X}(f)$.  Then by the definition of morphisms in $\cG \git \cG$, we have $\iota_{X}(f)=f'\circ g$ and hence $f=f' \circ g$.  Then there exists a triangle
\begin{center}
\begin{tikzpicture}[node distance=1cm,on grid]
\node (B) at (0,0) {$Y$};
\node (C) [right=of B] {$Z$};
\node (A) [below=of B,xshift=0.5cm,yshift=0cm] {$X$};
\draw[->] (B) -- node [above] {$g$} (C);
\draw[->] (B) -- node [left] {$f$} (A);
\draw[->] (C) -- node [right] {$f'$} (A);
\node at (-1,-0.5) {$\smor{g}\ =$};
\end{tikzpicture}
\end{center}
such that $\iota_{X}(\smor{g}\colon f \to f')=(f',g)$ as required.

Lastly, it is immediate that $\iota_{X}$ is injective on objects.
\end{proof}

Later, we will have need of the corresponding covariant functor $\iota_{X}^{*}\colon \op{(\cG \slash X)} \to \cG \git \cG$ given by
\begin{align*} & \iota_{X}(f\colon X \to Y) = f \\
  & \iota_{X}(\smor{g}\colon f' \to f) = (f',g)\colon f' \to f
\end{align*}
This functor is also fully faithful and injective on objects.

\subsubsection*{An analogue of the coset space}
  
Now let $\cH$ be a subgroupoid of $\cG$.  We will assume that $\cH$ is \emph{wide}, that is, $\Obj(\cH)=\Obj(\cG)$; this is a very mild restriction with respect to what follows, and is made so that additional clauses of the form ``when $X$ is also an object of $\cH$'' can be omitted. We will construct a groupoid corresponding to the action of $\cG$ on the right cosets of $\cH$.

In what follows, we use the ``source'' function $s\colon \Mor(\cG) \to \Obj(\cG)$, $s(f\colon X\to Y)=X$, and later the corresponding ``target'' function $t\colon \Mor(\cG)\to \Obj(\cG)$, $t(f\colon X\to Y)=Y$.

First, we define a relation $\sim_{\cH}$ on $\Mor(\cG)$ by
\[ (g\colon X\to Y) \sim_{\cH} (g'\colon X' \to Y') \iff (s(g)=s(g')\ \text{and}\ \exists\ h\in \Mor_{\cH}(Y,Y')\ \text{such that}\ g'=h\circ g). \]
It is straightforward to check that $\cH$ being a groupoid implies that $\sim_{\cH}$ is an equivalence relation. For $g\in \Mor(\cG)$, denote by $[g]$ its $\sim_{\cH}$-equivalence class.

Note that by the definition of $\sim_{\cH}$, $s\colon (\Mor(\cG)/\sim_{\cH})\to \Obj(\cG)$, $s([g])=s(g)$ is well-defined.

\begin{lemma} There is a contravariant functor $\rho_{\cH}\colon \cG \to \mathrm{Set}$, defined by
  \begin{align*} & \rho_{\cH}(X) = \{ [g] \mid s(g)=X \} \\
    & \rho_{\cH}(g\colon X \to Y) = \rho_{\cH}(g)\colon \rho_{\cH}(Y) \to \rho_{\cH}(X), \rho_{\cH}(g)([f])=[f\circ g]
  \end{align*}
\end{lemma}

  \begin{proof}
    We first check that $\rho_{\cH}(g)$, and hence $\rho_{\cH}$, is well-defined.  Assume that $[f]=[f']$, so there exists $h\in \Mor(\cH)$ such that $f'=h\circ f$.  Then $f' \circ g=h\circ f \circ g$, so $[f' \circ g]=[f \circ g]$. Then
    \[ \rho_{\cH}(g)([f])=[f\circ g]=[f'\circ g]=\rho_{\cH}(g)([f']) \]
    and $\rho_{\cH}(g)$ is well-defined for all $g$.

    We have $\rho_{\cH}(\id_{X}\colon X\to X)([f])=[f\circ \id_{X}]=[f]$ for all $f$ such that $s(f)=X$, so $\rho_{\cH}(\id_{X})=\id_{\rho_{\cH}(X)}$ as required.  Then given $g\colon X\to Y$, $g'\colon Y\to Z$ and $[f]\in \rho_{\cH}(Z)$, we have
    \[ \rho_{\cH}(g' \circ g)([f])=[f\circ g' \circ g]=\rho_{\cH}(g)([f\circ g'])=(\rho_{\cH}(g) \circ \rho_{\cH}(g'))([f]) \]
    and so $\rho_{\cH}$ is functorial.
  \end{proof}

  Next, we give the definition of the ``groupoid coset space''.

  \begin{definition} Let $\cG$ be a groupoid and $\cH$ a subgroupoid of $\cG$. Define the $\cH$-coset space action groupoid $(\cH \mathrel{:} \cG)\git \cG$ to be the groupoid with objects the elements of $\Mor(\cG)/\sim_{\cH}$ and morphisms
    \[ ([g],f)\colon [g] \to [g\circ f] \]
    defined when $t(f)=s([g])=s(g)$. Composition of morphisms is given by
    \[ ([g'],f')\circ ([g],f) = ([g],f\circ f') \]
    when this is defined, i.e.\ when $[g']=[g\circ f]$. We have $\id_{[g]}=([g],\id_{s([g])})=([g],\id_{s(g)})$.
    \end{definition}

  \begin{proposition} There is a full (covariant) functor that is surjective on objects, $\pi_{\cH}\colon \cG \git \cG \to (\cH \mathrel{:} \cG)\git \cG$ given by
    \begin{align*}
      & \pi_{\cH}(f)=[f] \\
      & \pi_{\cH}(g,f)=([g],f)
    \end{align*}
  \end{proposition}

  \begin{proof} Recall that in $\cG \git \cG$, we have $\id_{g}\colon g \to g$ given by $\id_{g}=(g,\id_{s(g)})$.  Then
    \[ \pi_{\cH}(g,\id_{s(g)})=([g],\id_{s(g)})=\id_{\pi_{\cH}(g)}. \]

    Consider the composable morphisms $(g\circ f,f')$, $(g,f)$ in $\cG \git \cG$.  We have
    \begin{align*} \pi_{\cH}((g\circ f,f') \circ (g,f)) & = \pi_{\cH}(g,f\circ f') \\
      & = ([g],f\circ f') \\
      & = ([g\circ f],f') \circ ([g],f) \\
      & \pi_{\cH}(g\circ f,f')\circ \pi_{\cH}(g,f)
    \end{align*}
    so that $\pi_{\cH}$ is (covariantly) functorial.

    Given $([g],f)\in \Mor((\cH \mathrel{:} \cG)\git \cG)$, we have $\pi_{\cH}(g,f)=([g],f)$, so that $\pi_{\cH}$ is full.
    
    That $\pi_{\cH}$ is surjective on objects is clear from the definitions.
  \end{proof}

  The ``source'' function can be upgraded to a (forgetful) functor on each of the categories considered thus far, as follows:
  \begin{itemize}
  \item $s\colon \op{(\cG \slash X)}\to \cG$, $s(f\colon Y\to X)=s(f)=Y$, $s(\smor{g}\colon f' \to f)=g^{-1}$, where $\smor{g}$ is the triangle as previously but now considered in $\op{(\cG \slash X)}$ and hence as a morphism from $f'$ to $f$;
  \item $s\colon \cG \git \cG \to \cG$, $s(g)=s(g)$, $s(g,f)=f^{-1}$;
  \item $s\colon (\cH \mathrel{:} \cG) \git \cG \to \cG$, $s([g])=s(g)$, $s([g],f)=f^{-1}$.
  \end{itemize}
Indeed, there is a commutative diagram of groupoids
\begin{center}
  \begin{tikzpicture}[node distance=2.5cm,on grid]
\node (B) at (0,0) {$\op{(\cG \slash X)}$};
\node (C) [right=of B] {$\cG \git \cG$};
\node (D) [right=of C] {$(\cH \mathrel{:} \cG) \git \cG$};
\node (A) [below=of C] {$\cG$};
\draw[->] (B) -- node [left,xshift=-0.25em] {$s$} (A);
\draw[->] (C) -- node [right,xshift=-0.1em] {$s$} (A);
\draw[->] (D) -- node [right] {$s$} (A);
\draw[->] (B) -- node [above] {$\iota_{X}^{*}$} (C);
\draw[->] (C) -- node [above] {$\pi_{\cH}$} (D);
\end{tikzpicture}
\end{center}

Let us consider the functor $\pi_{\cH}^{X}\defeq \pi_{\cH}\circ \iota_{X}^{*}$.  Explicitly, we have
\begin{align*}
  & \pi_{\cH}^{X}(f\colon Y \to X)=\pi_{\cH}(f)=[f] \\
  & \pi_{\cH}^{X}(\smor{g}:f' \to f)=\pi_{\cH}((f',g))=([f'],g)\colon [f'] \to [f]
\end{align*}
Let us define $(\cH \mathrel{:} \cG \slash X) \git \cG$ to be the full image of $\pi_{\cH}\circ \iota_{X}^{*}$.  Then $(\cH \mathrel{:} \cG \slash X) \git \cG$ is a connected subgroupoid of $(\cH \mathrel{:} \cG) \git \cG$; its fullness as a subcategory corresponds to being ``closed under the $\cG$-action'', or a ``$\cG$-submodule'' of the ``$\cG$-module'' $(\cH \mathrel{:} \cG) \git \cG$.

The groupoid $(\cH \mathrel{:} \cG) \git \cG$ is obtained by an equivalence relation on morphisms of $\cG$ coming from the subgroupoid $\cH$: somewhat loosely, morphisms in $\cH$ become equivalent to identity morphisms in $(\cH \mathrel{:} \cG)$.  The construction is done in a $\cG$-equivariant way, hence the ``coset space'' $(\cH \mathrel{:} \cG)$ still admits a $\cG$-action, giving rise to $(\cH \mathrel{:} \cG) \git \cG$.

The groupoid $(\cH \mathrel{:} \cG \slash X) \git \cG$ is a ``sliced'' version of this: we restrict attention to those $\sim_{\cH}$-equivalence classes of morphisms that contain at least one morphism whose codomain is the fixed object $X$. Just as with the usual slice category, by slicing we focus on a connected groupoid with a ``base point'' an object of interest to us, disposing of other components of $\cG$, of which there may be many.

Our $\cG$-action works by pre-composition, hence is compatible with the slicing at $X$.  The upshot of the above constructions is therefore that we have a groupoid whose objects are $\sim_{\cH}$-equivalence classes containing certain morphisms from our original groupoid, with a $\cG$-action by pre-composition.

Note that unless $\cH$ contains \emph{only} the identity morphisms, $\sim_{\cH}$-equivalence classes will necessarily contain morphisms with codomains not equal to $X$: $s([f])$ is well-defined (by construction) but $t([f])$ is in general not.

We conclude with a question:

\begin{question*} Consider $\cG$, $\cH$ groupoids, $\cF\colon \cG \to \cH$, $X\in \cG$ and the induced functor $\cF_{X}\colon \cG/X\to \cH/\cF X$. Does there exist a (wide) subgroupoid $\cK$ of $\cG$ such that $(\cK \mathrel{:} \cG \slash X) \git \cG$ gives a construction for groupoids analogous to the natural group action on the coset space of the kernel of a group homomorphism?
\end{question*}

If the answer is in the affirmative, one might hope for a version of the First Isomorphism Theorem in this setup.






\bibliographystyle{halpha}
\bibliography{biblio}\label{references}

\normalsize

\end{document}